\documentclass[12pt,reqno]{amsart}
\usepackage{amsmath, amsthm, amssymb}
\usepackage{hyperref}

\advance \topmargin by -\headheight
\advance \topmargin by -\headsep
     
\evensidemargin \oddsidemargin
\newtheorem{thm}{Theorem}[section]
\newtheorem{lemma}[thm]{Lemma}
\newtheorem{prop}[thm]{Proposition}

\theoremstyle{definition}

\theoremstyle{remark}

\numberwithin{equation}{section}

\newcommand{\mmod}[1]{\,\,\mathrm{mod}\,\,#1}

\def\alp{{\alpha}} 
\def\bet{{\beta}}  
 
\def\del{{\delta}}

\def\eps{\varepsilon}

\def\le{\leqslant} \def\ge{\geqslant}

\def\d{{\,{\rm d}}}

\def \bC {\mathbb C}

\def \bF {\mathbb F}

\def \bN {\mathbb N}

\def \bQ {\mathbb Q}
\def \bR {\mathbb R}
\def \bZ {\mathbb Z}

\def \bx {\mathbf x}

\def \fc {\mathfrak c}

\def \fm {\mathfrak m}

\def \fp {\mathfrak p}

\def \fC {\mathfrak C}

\def \fG {\mathfrak G}

\def \fM {\mathfrak M}

\def \cB {\mathcal B}
\def \cC {\mathcal C}

\def \cE {\mathcal E}

\def \cI {\mathcal I}
\def \cJ {\mathcal J}

\def \cS {\mathcal S}

\def \deg {\mathrm{deg}}

\begin{document}
\title[Averaging on thin sets of diagonal forms]{Averaging on thin sets of diagonal forms}
\author[Sam Chow]{Sam Chow}
\address{School of Mathematics, University of Bristol, University Walk, Clifton, Bristol BS8 1TW, United Kingdom}
\email{Sam.Chow@bristol.ac.uk}
\subjclass[2010]{11D72, 11E76, 11P55}
\keywords{Hardy-Littlewood method, Diophantine equations}
\thanks{}
\date{}
\begin{abstract} 
We investigate one-dimensional families of diagonal forms, considering the evolution of the asymptotic formula and error term. We then discuss properties of the average asymptotic formula obtained. The subsequent second moment analysis precipitates an effective means of computing $p$-adic densities of zeros for large primes $p$.
\end{abstract}
\maketitle

\section{Introduction}
\label{intro}

The study of families of integral forms has long been espoused by arithmetic geometers in order to understand fundamental differences in the prevalences of rational points of forms. The idea of counting zeros on average, however, seems to be recent, and has been quite successful in extracting information about typical behaviour in situations where the anomalies have not been classified \cite{Bre2004, BruD2012eq, Mad2010}.  Much of the inspiration for this paper is owed to recent work of Br\"udern and Dietmann \cite{BruD2012eq}. By averaging over all coefficients in a range far exceeding the box length $B$, they demonstrated the Hasse principle for almost all degree $k$ diagonal forms in just $3k+2$ variables. Our goal is to see what can be achieved using as little averaging as possible. 

Let $k \ge 3$ be an integer. Let $h_1,\ldots, h_s$ be polynomials of degree $d>1$ with integer coefficients, being pairwise relatively prime in $\bQ[t]$ and having no integer roots. In the case that $k$ is even, we assume that the leading coefficients of the $h_i$ do not all have the same sign. For positive integers $t \le B^\delta$, where $\delta$ is a small positive constant, we first estimate the number $N(B,t)$ of integral solutions $\mathbf{x} \in [-B,B]^s$ to
\begin{equation} \label{start} h_1(t)x_1^k + \ldots + h_s(t) x_s^k=0. \end{equation}

Suppose that either $s \ge 2^k + 1$ and $d \delta < 2^{1-k}$; or $s \ge 2k^2-2$ and $d \delta < 1/3$. Define
\begin{equation*}
\fG(t) = \sum_{q=1}^\infty q^{-s} \sum_{\substack{a = 1 \\ (a,q)=1}}^q  \prod_{i=1}^s  \sum_{m =1}^q e(ah_i(t) m^k/q). 
\end{equation*}

\begin{thm} \label{IndividualAsymptotic} There exist positive constants $C$ and $\eps$ such that if $B^\eps \ll t \le B^\delta$ then
\begin{equation} \label{IndividualAsymptoticEq}  N(B,t)t^d = C \fG(t)B^{s-k} + O(B^{s-k-\eps}). \end{equation}
\end{thm}

Theorem \ref{IndividualAsymptotic} provides an asymptotic formula for $N(B,t)$, since we know from \cite[Chapter 8]{Dav2005} that $\fG(t)$ is a positive real number for any integer $t$. Motivated by the estimate \eqref{IndividualAsymptoticEq}, we study the asymptotic behaviour of the corresponding weighted average. Let $T$ be a positive integer satisfying $B^\varepsilon \ll T \le B^\delta$.

\begin{thm} \label{average} There exist positive constants $K$ and $\eps$, independent of $T$, such that
\begin{equation*} T^{-1} \sum_{t \le T} N(B,t)t^d B^{k-s} - K \ll B^{-\eps}. \end{equation*}
\end{thm}

Let $C$ and $K$ be as in Theorems \ref{IndividualAsymptotic} and \ref{average}. Put
\begin{equation} \label{Gdef}
 \fG = \sum_{q=1}^\infty q^{-1-s} \sum_{t=1}^q \sum_{\substack{a = 1 \\ (a,q)=1}}^q  \prod_{i=1}^s
\sum_{m =1}^q e(ah_i(t) m^k/q).  
\end{equation}
We will see that $K = C \fG$. We will show in Lemma \ref{SSaverage} that $T^{-1} \sum_{t \le T} \fG(t)$ converges to $\fG$ as $T \to \infty$. If the variance of the singular series were zero then $KB^{s-k}t^{-d}$ would be a good approximation to $N(B,t)$ for almost all $t$.

\begin{prop} \label{dream} Assume that 
\begin{equation} \label{nilseq}
\lim_{T \to \infty} T^{-1} \sum_{t \le T} (\fG(t)-\fG)^2 = 0. 
\end{equation}
Then there exists a positive-valued arithmetic function $\rho(t)$, decreasing to 0, such that the following statement holds for almost all positive integers $t$: if $B \ge t^{1/\delta}$ then
\begin{equation*} | N(B,t)t^dB^{k-s} - K | < \rho(t).\end{equation*}
\end{prop}

This would be a beautiful estimate, well behaved and independent of any singular series. Unfortunately the variance is positive almost surely.
\begin{thm} \label{PosVar} Let $h_1, \ldots, h_s$ be as defined in the preamble to Theorem \ref{IndividualAsymptotic}, and assume that they are irreducible in $\bQ[t]$. Then $T^{-1} \sum_{t \le T} (\fG(t)-\fG)^2$ converges to a positive real number.
\end{thm}

The assumptions on $h_1,\ldots,h_s$ are mild. Indeed, almost all degree $d$ polynomials with integer coefficients are irreducible in $\bQ[t]$ and therefore have no integer root (see \cite{Kub2009} or \cite[p. 365]{PS1976}). Moreover, almost all $s$-tuples of such polynomials are pairwise relatively prime in $\bQ[t]$.

Theorem \ref{IndividualAsymptotic} ought to bear no surprises, but note that the coefficients are non-negligible in terms of the box size, and that the constants $C$ and $\eps$ do not depend on $t$. We omit the case $k=2$ for simplicity, as one would expect a Kloosterman \cite{Klo1927} or Heath-Brown \cite{HB1996} approach to produce the asymptotic formula \eqref{IndividualAsymptoticEq} with just $s \ge 4$ variables, subject to an assumption on $d \delta$. Theorem \ref{average} provides the estimate $KB^{s-k}$ for the weighted average $T^{-1} \sum_{t \le T} N(B,t)t^d$, with a uniform power saving in the error term. This is an elegant estimate, independent of any singular series and independent of $T$. Theorem \ref{PosVar} shows that the singular series $\fG(t)$ fluctuates not insignificantly about its mean value $\fG$. Our proof requires us to compare $p$-adic densities of zeros for large primes $p$, so we develop an effective method to compute these. To the author's knowledge, this is the first such method. Note that one could use the second part of \cite[Lemma 10.2]{Woo2013} to establish our results for any $s \ge 2k^2 - 2k$, but this would require a stronger assumption on $d \delta$.

We now elaborate briefly on previous related work, referring the reader to the introduction of \cite{BruD2012eq} for a more detailed account. Poonen and Voloch \cite{PV2003} considered the set of all integral forms of degree $k \ge 2$ in $s \ge 3$ variables, for a given $(k,s) \ne (2,3)$, demonstrating that a positive proportion of such forms are everywhere locally soluble. Browning and Dietmann \cite{BroD2009} proved an additive analogue of this result for $s \ge 4$, and also showed that almost all additive forms of degree $k \ge 2$ in $s=3$ variables admit no nontrivial integral solutions. It appears that average solution-counting was pioneered by Breyer in his doctoral thesis \cite{Bre2004}. Br\"udern and Dietmann subsequently strengthened his results to produce the aforementioned paper \cite{BruD2012eq}, and undertook parallel research for Diophantine inequalities in \cite{BruD2012ineq}.  Madlener applied similar techniques to two problems in his dissertation \cite{Mad2010}, in particular considering a thinner $s$-dimensional family where the coefficients are $\varphi_1(a_1), \ldots, \varphi_s(a_s)$ for fixed polynomials $\varphi_1, \ldots, \varphi_s$ of the same degree $d$, and showing that $O(k)$ variables suffice for almost all forms in this family to satisfy the Hasse principle. Note that in \cite{Bre2004}, \cite{BruD2012eq} and \cite{Mad2010} it is necessary for the coefficients to be much greater than $B$ in absolute value, in order to interchange the r\^oles of coefficients and variables.

For $3 \le k \le 6$ in Theorem \ref{IndividualAsymptotic}, our assumptions on $s$ and $d\delta$ present a trade-off between the number of variables needed and the allowed size of the coefficients $h_i(t)$. Assuming that $s \ge 2k^2-2$ we may, for any $k \ge 3$, allow coefficients of size $\asymp B^{1/3 - \eps}$. With $t^d = B^{1/3-\eps}$ and $m = \max |h_i(t)|$, Theorem \ref{IndividualAsymptotic} gives an asymptotic formula for the number of solutions $\mathbf{x}$ to equation \eqref{start} satisfying $0 < |\mathbf{x}| \ll m^{3+\eps}$ whenever $s \ge 2k^2-2$. If one is interested only in upper bounds for the height $|\bx|$ of a smallest nontrivial zero $\bx$ of a diagonal form in $s$ variables, with $m$ the maximum of the absolute values of the coefficients, then more can be said. Schmidt \cite{Sch1979} established the bound $m^{\eps}$ for $k$ odd and $s$ extremely large. Br\"udern \cite{Bru1996} achieved the bound $m^{8/3+\eps}$ for diagonal cubic forms in just nine variables. By averaging over all possible coefficients in a large range, Br\"udern and Dietmann \cite[Theorem 1.3]{BruD2012eq} obtained the impressive bound $m^{1/(s-2-k)}$ for almost all locally soluble diagonal forms in $s \ge 3k+2$ variables. Our averaging occurs over a much thinner set, and yields the expected asymptotic formula \eqref{IndividualAsymptoticEq}.

This paper is organised as follows. In \S \ref{CM} we establish Theorem \ref{IndividualAsymptotic} via the circle method. The singular integral is treated as in \cite[\S 3.1]{BruD2012eq}. Our results for $s \ge 2k^2-2$ and $d\delta < 1/3$ rely on recent work of Wooley \cite{Woo2013}, where efficient congruencing is used to show that $2k^2-2$ variables suffice to establish an asymptotic formula for diagonal forms under mild conditions on the coefficients. In \S \ref{AverageSS} we average the singular series by considering congruences with one extra variable, yielding Theorem \ref{average}. We conclude in \S \ref{SMA} by proving Proposition \ref{dream} and Theorem \ref{PosVar}. To show Theorem \ref{PosVar} we realise the second raw moment as an Euler product, leading us to compare $p$-adic densities. A simplest scenario is apparent, and a Chebotarev-type result (Lemma \ref{IrredPoly}) enables us to construct a `next simplest' scenario. The densities are understood via a lifting argument.

We adopt the convention that $\varepsilon$ denotes an arbitrarily small positive number, so that its value may differ between instances. The implicit constants in Vinogradov and Landau notation depend at most on $k,s,d,\delta,h_1,\ldots,h_s$ and $\varepsilon$. \emph{Almost all} means proportion 1, and similarly for \emph{almost surely}. Bold face will be used for vectors, for instance we shall abbreviate $(x_1,\ldots, x_s)$ to $\mathbf{x}$ and write \mbox{$|\mathbf{x}| = \max |x_i|$.} The letter $p$ is reserved for prime numbers. A \emph{$k$th power residue modulo $m$} is a $k$th power of some $x \in (\bZ / m \bZ)^\times$. Denote by $\bF_q$ the finite field of order $q$.

The author is very grateful towards Trevor Wooley for his enthusiastic supervision. Thanks also to Adam Morgan for his invaluable help with Lemmas \ref{CoprimePoly} and \ref{IrredPoly}.

\section{The circle method}
\label{CM}

We show, \emph{a fortiori}, that there exists $C>0$ such that
\begin{equation} \label{alt} 
 N(B,t)t^d B^{k-s} - C \fG(t) \ll B^{-\eps} + 1/t
\end{equation}
whenever $t \le B^\delta$ is a positive integer, since we will also need this to prove Theorem \ref{average} and Proposition \ref{dream}. Let $t \le B^\delta$ be a positive integer. Let 
\[ f(\gamma) = \sum_{|x| \le B} e(\gamma x^k),\]
and put $f_i(\gamma) = f(\gamma h_i(t))$ for $i=1,2,\ldots,s$. Fix an integer $\lambda$ such that $|h_i(t)| \le \lambda t^d$ uniformly in $i$ and $t$. 

We begin with the assumptions $s \ge 2^k + 1$ and $d \delta < 2^{1-k}$. If $a \ge 0$ and $q>0$ are integers, let $\fM(q,a)$ be the set of $\alpha \in (0,1)$ such that $|q \alpha - a| \le (2k \lambda B^{k-k/2^{k-1}})^{-1}$. Let $\fM$ be the (disjoint) union of the arcs $\fM(q,a)$ over $0 \le a \le q \le \lambda t^d B^{k/2^{k-1}}$ with $(a,q)=1$. Put $\fm = (0,1) \setminus \fM$. By orthogonality,
\begin{equation*} N(B,t) = \int_\fM f_1(\alpha) \cdots f_s(\alpha) \d\alpha + \int_\fm f_1(\alp) \cdots f_s(\alp) \d\alpha.\end{equation*}
The following is the crux of the minor arc analysis.

\begin{lemma} \label{crux}
Let $\alpha \in \fm$ and $c=h_i(t)$ for some $i$. Then
\begin{equation} \label{cruxbound} f(\alpha c) \ll B^{1-2^{1-k}+\varepsilon}. \end{equation}
\end{lemma}

\begin{proof}
Let $\fp$ be the set of real numbers $\beta$ such that if $q>0$ is coprime to $a$ and $|q \beta - a| \le (2kB^{k-1})^{-1}$ then $q > B$. First suppose that $\alpha c \in \fp$. By Dirichlet's approximation theorem \cite[Lemma 2.1]{Vau1997}, we may choose coprime $q > 0$ and $a$ such that $q \le 2kB^{k-1}$ and $|\alpha cq -a| \le (2kB^{k-1})^{-1}$. Since $\alpha c \in \fp$, we have $q>B$, and the bound \eqref{cruxbound} follows from Weyl's inequality \cite[Lemma 2.4]{Vau1997}.

Thus, we may assume that $\alpha c \notin \fp$. Now there exist coprime integers $q>0$ and $a$ such that $q \le B$ and $|\alpha cq - a| \le (2kB^{k-1})^{-1}$. By \cite[Theorem 4.1 and Lemma 4.6]{Vau1997}, we have
\begin{equation} \label{ch4eq} f(\alpha c) \ll q^{1/2+\varepsilon} + q^{-1/k} \min(B, |\alpha c - a/q|^{-1/k}), \end{equation}
so we may assume that $q \le B^{k/2^{k-1}}$. Write $a/c = \tilde a / \tilde c$ with $\tilde c > 0$ and $(\tilde a, \tilde c)=1$. Then $\tilde cq \le \lambda t^d B^{k/2^{k-1}}$, so the fact that $\alpha \in \fm$ gives 
\begin{equation*} |\tilde c q \alpha - \tilde a| > (2k \lambda B^{k-k/2^{k-1}})^{-1}. \end{equation*}
In particular $q|\alpha c - a/q| \gg (B^{k-k/2^{k-1}})^{-1}$ which, via the bound \eqref{ch4eq}, completes the proof.
\end{proof}
Periodicity and Hua's lemma \cite[Lemma 2.5]{Vau1997} yield
\begin{equation} \label{HuaResult} \int_0^1 |f_i(\alpha)|^{2^k} \d\alpha = \int_0^1 |f(\gamma)|^{2^k} \d\gamma \ll B^{2^k-k+\eps}, \end{equation}
and H\"older's inequality now gives
\begin{equation*} \int_0^1 |f_1(\alp) \cdots f_{2^k}(\alp)| \d \alpha \ll B^{2^k-k+\eps}. \end{equation*}
Combining this with Lemma \ref{crux} on the remaining variables, we get
\begin{equation} \label{minor} \int_\fm  f_1(\alp) \cdots f_s(\alp) \d\alpha
\ll B^{s-k+\eps -2^{1-k}(s-2^k)}. \end{equation}
This completes our treatment of the minor arcs $\fm$. 

Define $Q = \lambda t^d B^{k/2^{k-1}}$. When $\beta \in \bR$ and $P >0$, put
\[
v(\beta,P) = \int_{-P}^P e(\beta \gamma^k) \d\gamma
\]
and $v(\beta) = v(\beta,1)$. For $i=1,2,\ldots,s$, write $v_i(\beta,B) = v(\beta h_i(t),B)$ and $v_i(\beta) = v_i(\beta,1)$. Let 
\begin{equation} \label{SqaDef}
S(q,a) = \sum_{m =1}^q e(am^k/q)
\end{equation}
whenever $q > 0$ and $a$ are integers.

We begin by considering some $c= h_i(t)$ and $\alpha \in \fM(q,a)$, where $q \le Q$ and $(a,q)=1$. 
Write $c/q = \tilde c / \tilde q$ with $\tilde q > 0$ and $(\tilde c, \tilde q)=1$, and note that 
\begin{equation} \label{EasyIdentity} q^{-1} S(q,ac) = \tilde q^{-1} S(\tilde q, a\tilde c).\end{equation} 
Since $|q \alpha - a| \le (2k \lambda B^{k-k/2^{k-1}})^{-1}$ and $d\delta < 2^{1-k}$, we have
\begin{equation*} | \tilde q \alpha c - a \tilde c| \le t^d (2k B^{k-k/2^{k-1}})^{-1} \le (2kB^{k-1})^{-1}. \end{equation*}
Now \cite[Theorem 4.1]{Vau1997} yields
\begin{equation} \label{FirstEstPart} f(\alpha c) - \tilde q^{-1} S(\tilde q, a \tilde c) v(\alpha c - a\tilde c / \tilde q, B) \ll \tilde q^{1/2+ \eps}. \end{equation}

Let
\begin{equation*} V_i(\alpha) =  q^{-1} S(q, a h_i(t)) v_i(\alpha  - a / q, B) \qquad (1 \le i \le s) \end{equation*}
when $\alpha \in \fM(q,a) \subseteq \fM$, and let
\begin{equation*} X_s = \int_\fM V_1(\alpha) \cdots V_s(\alpha) \d \alpha.\end{equation*}
The identity \eqref{EasyIdentity} and the bound \eqref{FirstEstPart} give
\begin{equation*} 
\prod_{i=1}^s f_i(\alp) - \prod_{i=1}^s V_i(\alpha)  \ll 
Q^{s/2+\eps} + Q^{1/2+\eps} \max_i | f_i(\alp)|^{s-1}
\end{equation*}
so, using the bound \eqref{HuaResult},
\begin{equation*} 
\int_\fM f_1(\alp) \cdots f_s(\alp) \d\alpha - X_s \ll \cE_1+ \cE_2,
\end{equation*}
where
\[
\cE_1 = Q^{s/2+1} B^{k/2^{k-1}-k+\eps} \ll (t^d)^{s/2+1}  B^{(s/2+2)k/2^{k-1}-k+\eps}\]
and
\[
\cE_2 = Q^{1/2} B^{s-k-1+\eps} \ll t^{d/2}B^{k/2^k+s-k-1 +\eps}.
\]
In light of our assumptions $s \ge 2^k+1$ and $d \delta < 2^{1-k}$, we now have
\begin{equation} \label{FirstError}
\int_\fM f_1(\alp) \cdots f_s(\alp) \d\alpha - X_s \ll t^{d/2}B^{k/2^k+s-k-1 +\eps}.
\end{equation}

Next we estimate $X_s$. Let
\begin{equation} \label{GtqDef}
\fG(t,q) = \sum_{\substack{a = 1 \\ (a,q)=1}}^q \prod_{i=1}^s q^{-1} S(q,ah_i(t)),
\end{equation}
and note that
\begin{equation} \label{Gt} \fG(t) = \sum_{q=1}^\infty \fG(t,q).\end{equation}
Let
\[
J_s(q) = \int_{-T_1}^{T_1} v_1(\bet,B) \cdots v_s(\bet,B) \d\beta,\]
where $T_1 = (2k\lambda q B^{k-k/2^{k-1}})^{-1}$, so that
\begin{equation} \label{SingularProduct} X_s = \sum_{q \le Q} J_s(q) \fG(t,q).\end{equation}

First we consider $J_s(q)$. By a change of variables,
\begin{equation*} J_s(q) = B^{s-k} \int_{-T_2}^{T_2} v_1(\bet) \cdots v_s(\bet) \d\beta, \end{equation*}
where $T_2 = B^{k/2^{k-1}}/(2k \lambda q)$. By \cite[Lemma 3.1]{BruD2012eq}, the integral
\begin{equation*} I(t)= \int_{-\infty}^\infty  v_1(\bet) \cdots v_s(\bet) \d\beta \end{equation*}
is a positive real number. The classical bound \cite[Theorem 7.3]{Vau1997} gives
\begin{equation} \label{intermezzo}
v_i(\beta) \ll |\beta h_i(t)|^{-1/k} \le |\beta|^{-1/k}.
\end{equation}
Hence
\begin{equation*}
J_s(q)B^{k-s}-I(t) \ll \Bigl( \prod_{i=1}^s |h_i(t)|^{-1/k}\Bigr) \int_{T_2}^\infty \beta^{-s/k} \d\beta,
\end{equation*}
so
\begin{equation} \label{SecondEst}
t^d( J_s(q)B^{k-s}-I(t) ) \ll (Q/q)^{1-s/k}.
\end{equation}

Next we consider $\fG(t,q)$. By the identity \eqref{EasyIdentity} and \cite[Theorem 4.2]{Vau1997}, 
\begin{equation} \label{I_2_1} \prod_{i=1}^s q^{-1} S(q,ah_i(t)) \ll q^{-s/k}\prod_{i=1}^s(h_i(t),q)^{1/k}.\end{equation}
We now exploit the fact that $h_1, \ldots, h_s$ are pairwise relatively prime in $\bQ[t]$.

\begin{lemma} \label{CoprimePoly} Let $q$ and $t$ be positive integers. Then
\begin{equation} \label{I_2_2} \prod_{i=1}^s (h_i(t),q) \ll q \end{equation}
uniformly in $t$ and $q$.
\end{lemma}

\begin{proof}
For each pair $h_i, h_j$ with $i \ne j$, perform Euclid's algorithm in $\bQ[t]$ and clear denominators to give a positive integer $m_{ij}$ and polynomials $\varphi_{ij}, \psi_{ij} \in \bZ[t]$ such that 
\begin{equation*} h_i\varphi_{ij} + h_j \psi_{ij} = m_{ij}. \end{equation*}
Then
\begin{equation} \label{PolyGCDbound}(h_i(t), h_j(t)) \le m_{ij}, \end{equation}
and this bound depends only on the polynomials $h_i$ and $h_j$. Note that if $x,y \in \bN$ then $(x,q)(y,q) \le (xy,q)(x,y)$. This shows by induction on $n \le s$ that
\[ \prod_{i=1}^n (h_i(t),q) \le \Big(\prod_{i=1}^n h_i(t), q \Big) \prod_{1 \le i < j \le n} (h_i(t), h_j(t)).\]
Hence,
\[ \prod_{i=1}^s (h_i(t),q) \le q \prod_{1 \le i < j \le s} m_{ij}.\]
\end{proof}

The inequalities \eqref{I_2_1} and \eqref{I_2_2} yield
\begin{equation} \label{SingularBound1} \fG(t,q) \ll q^{1+(1-s)/k} \end{equation}
so, by equation \eqref{SingularProduct} and the inequality \eqref{SecondEst},
\begin{equation} \label{SecondError}
t^d \Big( X_s B^{k-s} - I(t) \sum_{q \le Q} \fG(t,q)  \Big) \ll 
Q^{(2k+1-s)/k}.
\end{equation}
The identity \eqref{Gt} and the bound \eqref{SingularBound1} give
\begin{equation} \label{SingularSeriesBounded} \fG(t) \ll 1 \end{equation}
and
\begin{equation} \label{ThirdError}
I(t) \sum_{q \le Q} \fG(t,q)  - I(t) \fG(t)  \ll I(t)Q^{(2k+1-s)/k}.
\end{equation}

The error bounds \eqref{minor}, \eqref{FirstError}, \eqref{SecondError} and \eqref{ThirdError} yield
\begin{equation} \label{IndividualError}
t^d (N(B,t)B^{k-s} - I(t) \fG(t) ) \ll E+I(t) t^dQ^{(2k+1-s)/k},
\end{equation}
where
\[ E= t^d B^{\eps - 2^{1-k}(s-2^k)} +  t^{3d/2} B^{k/2^k-1+\eps} + Q^{(2k+1-s)/k}. \]

\begin{lemma} \label{SingularIntegral} There exists a constant $C>0$ such that 
\begin{equation*} I(t)t^d = C + O(1/t)
\end{equation*}
for positive integers $t$.
\end{lemma}

\begin{proof} We may assume that $t$ is large in terms of $h_1, \ldots, h_s$. First assume that $k$ is odd. By a change of variables,
\begin{equation*} I(t) = k^{-s} \int_{-\infty}^\infty  \int_{(-1,1)^s} (\zeta_1 \cdots \zeta_s)^{1/k-1} e(\beta {\bf h} \cdot \boldsymbol \zeta) \d  \boldsymbol \zeta \d\beta, \end{equation*}
where $\mathbf{h} = (h_1(t),\ldots, h_s(t))$. The strategy is to change the order of integration and use Fourier inversion. It transpires that it is of key importance that a certain function yet to appear is of bounded variation. We begin by tweaking the integral so that this will follow easily. For $0 < \rho < 1$ put
\begin{equation*} W(\rho) = \{ \boldsymbol \zeta \in (-1,1)^s: | \zeta_i |> \rho \text{ for } i=1,2,\ldots,s \}
\end{equation*}
and
\begin{equation*}  U(\rho) = (-1,1)^s \setminus W(\rho).\end{equation*}
For $S \subseteq (-1,1)^s$, let
\begin{equation*} I_S  = \int_{-\infty}^\infty \int_S (\zeta_1 \cdots \zeta_s)^{1/k-1} e(\beta {\bf h} \cdot \boldsymbol \zeta) \d \boldsymbol \zeta \d\beta,\end{equation*}
so that $k^s I(t) = I_{U(\rho)} + I_{W(\rho)}$. The bounds
\begin{equation*}
\int_{-\rho}^{\rho} \zeta_i^{1/k-1}\d\zeta_i \ll \rho^{1/k}
\end{equation*}
and \eqref{intermezzo} imply that $I_{U(\rho)} \ll \rho^{1/k} \to 0$ as $\rho \to 0^+$, so
\begin{equation} \label{IntegralLimit} k^s I(t) = \lim_{\rho \to 0^+} I_{W(\rho)}.\end{equation}

By Fubini's theorem we deduce
\begin{equation*}
I_{W(\rho)} = \lim_{X \to \infty} \int_{W(\rho)} (\zeta_1 \cdots \zeta_s)^{1/k -1} \frac{\sin(2\pi X  {\bf h} \cdot \boldsymbol \zeta)} {\pi  {\bf h} \cdot \boldsymbol \zeta} \d \boldsymbol {\zeta}.
\end{equation*}
Write $a_s = |h_s(t)|$,
\begin{equation*} \boldsymbol \zeta' = (\zeta_1,\ldots,\zeta_{s-1}), 
\qquad Y = \sum_{i=1}^{s-1}h_i(t)\zeta_i,\qquad \cC_\rho = (-1, -\rho) \cup (\rho,1) \end{equation*}
and \[
\cC_\rho'  = (Y- a_s, Y- a_s \rho) \cup (Y+a_s \rho, Y+a_s).
\]
Changing variables from $\zeta_s$ to $u = {\bf h} \cdot \boldsymbol \zeta$ shows that $a_s^{1/k} I_{W(\rho)}$ equals
\begin{equation*} \lim_{X \to \infty} \int_{\cC_\rho^{s-1}} (\zeta_1 \cdots \zeta_{s-1})^{1/k -1}
  \int_{\cC_\rho'} (u-Y)^{1/k -1} \frac{\sin(2\pi X u)}{\pi u} \d u \d \boldsymbol \zeta '.
\end{equation*}
For real numbers $u$, let
\begin{align*}
\cB_\rho (u) &= \{ \boldsymbol \zeta' \in \cC_\rho^{s-1}: a_s \rho < |u-Y| < a_s \}, \\
g_u(\boldsymbol \zeta') &= (\zeta_1 \cdots \zeta_{s-1})^{1/k-1} (u-Y)^{1/k-1}
\end{align*}
and \[
V_\rho(u) = \int_{\cB_\rho (u)} g_u(\boldsymbol \zeta') \d \boldsymbol \zeta',
\]
so that
\begin{equation*}
a_s^{1/k} I_{W(\rho)} = \lim_{X \to \infty} \int_{-\infty}^\infty V_\rho(u) \frac{\sin(2\pi X u)}{\pi u} \d u.
\end{equation*}

The function $V_\rho(\cdot)$ is compactly supported. To show that it is of bounded variation, we show that the gradient
$(V_\rho(u) - V_\rho(v))/(u-v)$
is bounded for $u\ne v$. Since the numerator is bounded, it suffices to consider $0 < |u-v| < a_s \rho / 2$. Here
\[ \frac{V_\rho(u) - V_\rho(v)}{u-v} = \cI + \cJ, \]
where
\begin{equation*}
\cI = \int_{\cB_\rho(u)} (\zeta_1 \cdots \zeta_{s-1})^{1/k-1} \frac{(u-Y)^{1/k-1} - (v-Y)^{1/k-1}}{u-v} \d \boldsymbol \zeta'
\end{equation*}
and 
\begin{equation*}
\cJ = \frac{ \{ \int_{\cB_\rho(u) \setminus \cB_\rho(v)} -  \int_{\cB_\rho(v) \setminus \cB_\rho(u)} \} g_v(\boldsymbol\zeta') \d \boldsymbol \zeta'} {u-v}.
\end{equation*}
The integral $\cI$ is bounded because the integrand is bounded, while the quantity $\cJ$ is bounded because the integrand is bounded and the volume of integration is $O(|u-v|)$. The function $V_\rho(\cdot)$ is therefore of bounded variation.

Using the Fourier inversion theorem,
\begin{align} \notag
a_s^{1/k} I_{W(\rho)}
&= \lim_{X \to \infty} \int_{-\infty}^\infty V_\rho(u) \int_{-X}^X e(xu) \d x  \d u
\\ \label{FIT} &= \int_{-\infty}^\infty \int_{-\infty}^\infty e(xu)V_\rho(u) \d u \d x = V_\rho(0).
\end{align}
Let
\begin{equation*}
\cB(0) = \{ \boldsymbol \zeta' \in (-1,1)^{s-1}: |Y| < a_s\}
\end{equation*}
and
\begin{equation*} V(0) = \int_{\cB(0)} g_0(\boldsymbol \zeta') \d \boldsymbol \zeta'. \end{equation*}
By the monotone convergence theorem \mbox{$V_\rho(0) \to V(0)$} as $\rho \to 0^+$. Equations \eqref{IntegralLimit} and \eqref{FIT} now give
\begin{equation} \label{V0} k^s a_s^{1/k} I(t) = V(0). \end{equation}

Next we consider $V(0) = \int_{\cB(0)} (\zeta_1 \cdots \zeta_{s-1}Y)^{1/k-1} \d\boldsymbol \zeta'$. For $i=1,2,\ldots, s$, let $c_i$ denote the leading coefficient of the polynomial $h_i$. Put
\begin{equation*}
\eta_i = 
\begin{cases}
h_i(t) / (c_it^d) \cdot \zeta_i , & i=1,2,\ldots,s-1 \\
-Y/(c_st^d), &i=s.\end{cases}
\end{equation*}
Since $\eta_i = (1+O(1/t)) \zeta_i$ for $i=1,2,\ldots,s-1$ and $a_s = (1+O(1/t)) |c_s|t^d$, we have
\begin{equation*}
\frac{(\zeta_1 \cdots \zeta_{s-1}Y)^{1/k-1}}{(\eta_1 \cdots \eta_s)^{1/k-1}}
 =(1+O(1/t))(c_st^d)^{1/k-1}
= \frac{(1+O(1/t))a_s^{1/k}}{|c_s|t^d}.
\end{equation*}
Observe that
\begin{equation} \label{EtaS} 
\eta_s = -c_s^{-1}(c_1\eta_1+\ldots+c_{s-1}\eta_{s-1}).
\end{equation}
We recall equation \eqref{V0} and change variables from $\zeta_1,\ldots,\zeta_{s-1}$ to $\eta_1,\ldots,\eta_{s-1}$. The Jacobian determinant is 
\[ \prod_{i=1}^{s-1} \frac{\partial \zeta_i}{\partial \eta_i} = 1+O(1/t),\]
so
\begin{equation} \label{cp}
k^s t^d I(t) = \frac{1+O(1/t)}{|c_s|} \int_{R_1} (\eta_1 \cdots \eta_s)^{1/k-1} \d\eta_1 \cdots \d\eta_{s-1},
\end{equation}
where 
\begin{equation*} R_1 = \{ (\eta_1, \ldots,  \eta_{s-1}) : |\eta_i| < h_i(t)/ (c_it^d) \text{ for } i=1,2,\ldots,s\}.
\end{equation*}
Note that $h_i(t)/(c_it^d) > 0$ for all $i$, since $t$ is large. 

Recall equation \eqref{EtaS}, and let 
\begin{equation*} R_2 = \{ (\eta_1, \ldots, \eta_{s-1}) \in (-1,1)^{s-1} : |\eta_s| <1 \}.\end{equation*}
Next we establish that
\begin{equation} \label{SymDif}
\int_{\Delta} (\eta_1 \cdots \eta_s)^{1/k-1} \d\eta_1 \cdots \d\eta_{s-1} \ll t^{-1},
\end{equation}
where $\Delta = R_1 \Delta R_2$ is the symmetric difference. For any $(\eta_1, \ldots, \eta_{s-1}) \in \Delta$, there exist $i,j \in \{1,2, \ldots, s \}$ satisfying
\begin{equation} \label{subdivide}
i \ne j, \qquad |\eta_i| - 1 \ll t^{-1}, \qquad \eta_j \gg 1.
\end{equation}
Hence $\Delta = \cup_j \Delta_j$, where $\Delta_j$ is the set of $(\eta_1, \ldots, \eta_{s-1}) \in \Delta$ such that the conditions \eqref{subdivide} are met for some $i$. Changing variables from $\eta_j$ to $\eta_s$ (if $j \ne s$) shows that
\[ 
\int_{\Delta_j} (\eta_1 \cdots \eta_s)^{1/k-1} \d\eta_1 \cdots \d\eta_{s-1} \ll t^{-1}
\]
for $j=1,2,\ldots,s$, which proves the inequality \eqref{SymDif}.

Recalling equation \eqref{cp}, we now have
\begin{equation*}
I(t)t^d = \frac{1+O(1/t)}{k^s|c_s|} \int_{R_2} (\eta_1 \cdots \eta_s)^{1/k-1} \d\eta_1 \cdots \d\eta_{s-1} = C+O(1/t),
\end{equation*}
where we define the positive constant
\begin{equation*}
C =  \frac{1}{k^s |c_s|} \int_{R_2} (\eta_1 \cdots \eta_s)^{1/k-1} \d\eta_1 \cdots \d\eta_{s-1}.
\end{equation*}
To see that $C<\infty$, observe that the same procedure with $c_1,\ldots, c_s$ in place of $h_1(t), \ldots, h_s(t)$ yields
\begin{equation*} C = \int_{-\infty}^{\infty} 
v(\beta c_1) \cdots v(\bet c_s) \d \beta \end{equation*}
which, by \cite[Theorem 7.3]{Vau1997}, is well defined. This establishes Lemma \ref{SingularIntegral} in the case that $k$ is odd.

Suppose now that $k$ is even. We follow closely the case where $k$ is odd, so the notation is the same unless otherwise indicated. We now have
\begin{equation*}
I(t) = (2/k)^s \int_{-\infty}^\infty \int_{(0,1)^s} (\zeta_1 \cdots \zeta_s)^{1/k-1} e(\beta {\bf h} \cdot \boldsymbol \zeta) \d \boldsymbol \zeta \d\beta.
\end{equation*}
When $0<\rho<1$, let $W(\rho) = (\rho,1)^s$ and $U(\rho) = (0,1)^s \setminus W(\rho)$. For $S \subseteq (0,1)^s$, put
\begin{equation*} I_S  = \int_{-\infty}^\infty \int_S(\zeta_1 \cdots \zeta_s)^{1/k-1} e(\beta {\bf h} \cdot \boldsymbol \zeta) \d \boldsymbol \zeta\d\beta,\end{equation*}
so that $(k/2)^s I(t) = I_{U(\rho)} + I_{W(\rho)}$. Again $I_{U(\rho)} \to 0$ as $\rho \to 0^+$, so
\begin{equation*} (k/2)^s I(t)  = \lim_{\rho \to 0^+} I_{W(\rho)}.\end{equation*}

For $i=1,2,\ldots,s$, let $c_i$ denote the leading coefficient of $h_i$, and without loss of generality $c_s>0$. Put $a_s = h_s(t)$. Since $t$ is large, we have $a_s > 0$. We derive
\begin{equation*}
a_s^{1/k}I_{W(\rho)} = \lim_{X\to \infty} \int_{\infty}^\infty \frac{\sin(2\pi X u)}{\pi u}  \int_{\cB_{\rho}(u)}g_u( \boldsymbol \zeta')\d \boldsymbol \zeta' \d u,
\end{equation*}
where $\boldsymbol \zeta' = (\zeta_1,\ldots,\zeta_{s-1})$, $Y= \sum_{i=1}^{s-1}h_i(t) \zeta_i$,
\[
g_u(\boldsymbol \zeta') = (\zeta_1 \cdots \zeta_{s-1})^{1/k-1} (u-Y)^{1/k-1} \]
and \[
\cB_\rho(u) =  \{ \boldsymbol \zeta' \in (\rho,1)^{s-1}: \rho a_s < u-Y < a_s \}.
\]
Continuing to follow the case where $k$ is odd, we deduce that
\begin{equation*} 
(k/2)^s a_s^{1/k} I(t) = \int_{\cB(0)} (\zeta_1 \cdots \zeta_{s-1}(-Y))^{1/k-1} \d\boldsymbol \zeta', 
\end{equation*}
where 
\begin{equation*}
\cB(0) = \big\{ \boldsymbol \zeta' \in (0,1)^{s-1}: 0 < -Y < a_s \big\}.
\end{equation*}

With $\eta_s = -c_s^{-1}(c_1\eta_1 + \ldots + c_{s-1}\eta_{s-1})$, let
\[
 R_1 = \{ (\eta_1 ,\ldots, \eta_{s-1}) : 0< \eta_i < h_i(t)/ (c_it^d) \text{ for } i=1,2,\ldots,s\}
\]
and 
\[
R_2 = \{ (\eta_1, \ldots, \eta_{s-1}) \in (0,1)^{s-1} :  0< \eta_s <1 \}.
\]
By changing variables with $\eta_i = h_i(t)/ (c_it^d) \cdot \zeta_i$ for $i=1,2,\ldots,s-1$, we deduce that
\begin{equation*}
(k/2)^s t^d  I(t) =  c_s^{-1}(1+O(1/t)) \int_{R_1} (\eta_1 \cdots \eta_s)^{1/k-1} \d\eta_1 \cdots \d\eta_{s-1},
\end{equation*}
and hence
\begin{equation*}
I(t) t^d = C+O(1/t),
\end{equation*}
where 
\begin{equation*}
C =  c_s^{-1}(2/k)^s \int_{R_2} (\eta_1 \cdots \eta_s)^{1/k-1} \d\eta_1 \cdots \d\eta_{s-1}.
\end{equation*}
The $c_i$ are not all of the same sign, so $R_2$ has nonempty interior, so $C>0$, and $C<\infty$ because, as in the odd case,
\begin{equation*} C = \int_{-\infty}^{\infty} 
v(\beta c_1) \cdots v(\bet c_s) \d \beta. \end{equation*}
\end{proof}

Combining Lemma \ref{SingularIntegral} with the inequalities \eqref{SingularSeriesBounded} and \eqref{IndividualError} gives 
\begin{equation*}
N(B,t)t^dB^{k-s} - C\fG(t)  \ll E + 1/t,
\end{equation*}
which yields the error bound \eqref{alt} subject to the hypotheses $s \ge 2^k + 1$ and $d \delta < 2^{1-k}$. 

It remains to discuss the case where \begin{equation} \label{NewAssumptions}
s \ge 2k^2-2 \qquad \text{and} \qquad d\delta < 1/3. \end{equation}
If $a \ge 0$ and $q>0$ are integers, let $\fM(q,a)$ be the set of $\alpha \in (0,1)$ such that $|q\alpha-a| \le (2k B^{k-1})^{-1}$. Let $\fM$ be the (disjoint) union of the arcs $\fM(q,a)$ over $0 \le a \le q \le Q$ with $(a,q)=1$, where
\begin{equation*}
Q = \lambda t^d B / (2k),
\end{equation*}
and let $\fm = (0,1) \setminus \fM$. 

Let $\fp$ be the set of $\beta \in (0,1)$ such that if $(a,q) =1 $ and $|q \beta - a| \le (2kB^{k-1})^{-1}$ then $q > B/(2k)$. For $i=1,\ldots,s$ let $\fp_i$ be the set of $\beta \in (0,|h_i(t)|)$ such that if $(a,q) =1$ and $|q \beta - a| \le (2kB^{k-1})^{-1}$ then $q > B/(2k)$. It is routinely verified that if $\alpha \in \fm$ then 
\begin{equation*} \alpha |h_i(t)| \in \fp_i =  \fp \cup (1+\fp) \cup \ldots \cup (|h_i(t)| - 1 + \fp). \end{equation*}
Thus, by periodicity and \cite[Theorem 10.1]{Woo2013} we have, for $i=1,2,\ldots, s$, 
\[
\int_\fm |f_i(\alp)|^s \d \alpha \le  \int_{\fp} |f(\beta)|^s \d\beta \ll B^{s-k-1+\eps}.
\]
Now, by H\"older's inequality,
\begin{equation*} \int_\fm f_1(\alp) \cdots f_s(\alp) \d\alp \ll B^{s-k-1+\eps}. \end{equation*}

For the major arcs we follow the previous analysis, \emph{mutatis mutandis}, so the notation is the same unless otherwise indicated. Using the definitions preceding equation \eqref{EasyIdentity}, we have $| \tilde q \alpha c - a \tilde c| \le \lambda t^d /(2kB^{k-1})$, so \cite[Theorem 4.1]{Vau1997} and equation \eqref{EasyIdentity} give
\begin{equation*} f(\alpha c) - q^{-1}S(q,ac) v(\alpha c- ac/q) \ll \tilde q^{1/2+\eps} (1+t^d B/\tilde q)^{1/2} \ll Q^{1/2+\eps}. \end{equation*}
Invoking \cite[Lemma 10.2]{Woo2013} and using periodicity now yields
\begin{equation*} 
\int_\fM f_1(\alp) \cdots f_s(\alp) \d\alpha - X_s \ll \cE_1+ \cE_2,
\end{equation*}
where
\[
\cE_1 = Q^{s/2+1} B^{1-k+\eps} \ll (t^d)^{s/2+1} B^{s/2+2-k+\eps} \]
and
\[
\cE_2 = Q^{1/2} B^{s-k-1+\eps} \ll t^{d/2}B^{s-k-1/2+\eps}.
\]
In light of our assumptions \eqref{NewAssumptions}, we now have
\[
\int_\fM f_1(\alp) \cdots f_s(\alp) \d\alpha - X_s  \ll t^{d/2}B^{s-k-1/2+\eps}.
\]

This time
\begin{equation*} J_s(q) = B^{s-k} \int_{-B/(2kq)}^{B/(2kq)} v_1(\beta) \cdots v_s(\beta) \d \beta, \end{equation*}
and we again have the error bound \eqref{SecondEst}, but with $Q = \lambda t^d B / (2k)$. The remainder of the analysis is identical, and we conclude that 
\begin{equation*} N(B,t)t^dB^{k-s} - C\fG(t) \ll t^{3d/2}B^{\eps-1/2} + Q^{(2k+1-s)/k} +1/t, \end{equation*}
with $C > 0$ as in Lemma \ref{SingularIntegral}. We now have the error bound \eqref{alt} under the hypotheses \eqref{NewAssumptions} also, which completes the proof of Theorem \ref{IndividualAsymptotic}.

\section{Averaging on the singular series}
\label{AverageSS}
Put
\begin{equation} \label{Sdef}
S(q) = q^{-1} \sum_{t=1}^q \fG(t,q).
\end{equation}
Recalling equations \eqref{Gdef}, \eqref{SqaDef} and \eqref{GtqDef}, we note that 
\[ \fG = \sum_{q=1}^\infty S(q).\]
The inequality \eqref{SingularBound1} ensures that $\sum_{q=1}^\infty S(q)$ converges absolutely, so $\fG \in \bC$. The following is the key to Theorem \ref{average}.
\begin{lemma} \label{SSaverage} The constant $\fG$ is a positive real number, and
\begin{equation} \label{SSaverageEq} T^{-1} \sum_{t \le T} \fG(t) = \fG + O(T^{\eps-1}+T^{\varepsilon+2-(s-1)/k}). \end{equation}
\end{lemma}

\begin{proof} We begin by demonstrating the asymptotic formula \eqref{SSaverageEq}. Recall equations \eqref{SqaDef} and \eqref{GtqDef}, and recall that $T$ is a positive integer. For a given positive integer $q$, let $r = r(q,T)$ denote the remainder when $T$ is divided by $q$. By periodicity modulo $q$, we have
\begin{equation*} \sum_{t=r+1}^T \fG(t,q) = \frac{T -r}q \sum_{t=1}^q \fG(t,q).\end{equation*}
Recalling equation \eqref{Gt}, we now have
\begin{align*}T^{-1} \sum_{t \le T} \fG(t) - \fG &=  
\sum_{q=1}^\infty T^{-1} \Big( \sum_{t\le T} \fG(t,q) - \frac Tq \sum_{t=1}^q \fG(t,q) \Big) \\
&= \sum_{q=1}^\infty T^{-1} \Big( \sum_{t=1}^r \fG(t,q) - \frac rq \sum_{t=1}^q \fG(t,q) \Big).
\end{align*}

Noting that $r(q,T)=T$ when $q>T$, the inequality \eqref{SingularBound1} now gives
\begin{equation*}T^{-1} \sum_{t \le T} \fG(t) - \fG
\ll T^{-1} \sum_{q \le T} r	q^{1+(1-s)/k}+ \sum_{q>T}  q^{1+(1-s)/k}.
\end{equation*}
This yields equation \eqref{SSaverageEq}, since
\begin{align} \notag
\sum_{q \le T}r q^{1+(1-s)/k} &= \sum_{r=0}^{T-1} r \sum_{\substack{q>r \\ q | (T-r) }} q^{1+(1-s)/k}
\ll T^\eps \sum_{r=1}^{T-1} r^{2+(1-s)/k} \\
&\label{calculation} \ll T^{\eps} + T^{\varepsilon+3-(s-1)/k}.
\end{align}

Next we establish that $\fG$ is a positive real number. By following the proof of \cite[Lemma 2.11]{Vau1997}, we deduce that $S(\cdot)$ is multiplicative. Since $S(\cdot)$ is multiplicative and $\fG = \sum_{q=1}^\infty S(q)$ converges absolutely, we have the absolutely convergent Euler product
\begin{equation} \label{EulerProd} 
\fG = \prod_p T(p), 
\end{equation}
where
\[
T(p) = \sum_{h=0}^\infty S(p^h). 
\]

In order to prove that $\fG > 0$ it remains to show that $T(p) > 0$ for each prime $p$. Let
\begin{equation*}
M(q) = q^{-1} \sum_{r =1}^q \sum_{t=1}^q \sum_{\mathbf{x}} \prod_{i=1}^s e( q^{-1}r h_i(t)x_i^k ),
\end{equation*}
where the inner summation is over $x_1, \ldots, x_s$ modulo $q$. By orthogonality, $M(q)$ counts integer solutions to
\begin{equation} \label{ModEqn} h_1(t)x_1^k + \ldots +h_s(t) x_s^k \equiv 0 \mmod q \end{equation}
with $1 \le x_1, \ldots, x_s, t \le q$. Writing $r/q = a/d$ with $d>0$ and $(a,d)=1$ yields
\begin{equation*}
M(q) = q^{-1} \sum_{d|q} \sum_{\substack{a=1 \\ (a,d)=1}}^d (q/d)^{s+1}
\sum_{t =1}^d \sum_{\mathbf{x}} \prod_{i=1}^s e(  d^{-1}a h_i(t)x_i^k ), \end{equation*}
where the inner summation is over $x_1, \ldots, x_s$ modulo $d$. Thus, upon recalling equations \eqref{SqaDef}, \eqref{GtqDef} and \eqref{Sdef}, we have
\begin{equation*} q^{-s}M(q) = \sum_{d|q} S(d). \end{equation*} 
In particular,
\begin{equation} \label{TpAlt}
T(p) = \lim_{r \to \infty} p^{-rs}M(p^r).
\end{equation}

Let $p^\tau \| k$, and define
\begin{equation} \label{GammaDef} \gamma =  \gamma(p) = 
\begin{cases} \tau +1, &\text{if }p>2 \text{ or } (p,\tau) = (2,0) \\
\tau+2, &\text{if }p =2 \text{ and } \tau > 0.
\end{cases}
\end{equation}
By the discussion on \cite[p. 22]{Vau1997}, the $k$th power residues modulo $p^\gamma$ are $k$th power residues modulo every power of $p$. Choose a positive integer $m$ such that \mbox{$p^m > \max_i |h_i(1)|$,} and let $r \ge \gamma+m$. There are $p^{r-m}$ possibilities for $t \mmod p^r$ that are congruent to 1 modulo $p^m$. Thus, in order to show that $T(p)>0$, it suffices to show that for such values of $t$ and $r$ there are at least $p^{(s-1)(r-\gamma-m)}$ solutions $\mathbf{x}$ modulo $p^r$ to the congruence
\begin{equation} \label{pAdicEqn}
h_1(t)x_1^k + \ldots + h_s(t)x_s^k \equiv 0 \mmod p^r.
\end{equation}

For $i=1,2,\ldots,s$, put $a_i = h_i(t)$ and note that $p^m$ does not divide $a_i$. The illustrious work of Davenport and Lewis \cite{DL1963} tells us that $s$ is large enough to guarantee a solution $\mathbf{y}$ modulo $p^{\gamma+m}$ to the congruence
\begin{equation*}
a_1y_1^k + \ldots + a_sy_s^k \equiv 0 \mmod p^{\gamma+m}
\end{equation*}
such that $p$ does not divide $\gcd(y_1,\ldots,y_s)$. Without loss of generality $p$ does not divide $y_1$, and we show that this solution lifts to the requisite $p^{(s-1)(r-\gamma-m)}$ solutions $\mathbf{x}$ modulo $p^r$ to the congruence \eqref{pAdicEqn}. There are $p^{(s-1)(r-\gamma-m)}$ ways to choose $x_2,\ldots, x_s$ modulo $p^r$ such that 
\begin{equation*} x_i \equiv y_i \mmod p^{\gamma + m} \qquad (i=2, \ldots,s),\end{equation*}
so it remains to show that, given any such $x_2,\ldots,x_s$, there exists $x_1$ such that the congruence \eqref{pAdicEqn} is satisfied. 

Note that
\begin{equation*} -a_1 y_1^k \equiv \sum_{i=2}^s a_i x_i^k \mmod p^{\gamma+m}.
\end{equation*}
Let $p^\alpha \| a_1$, and let $c$ be the multiplicative inverse of $-a_1p^{-\alpha}$ modulo $p^r$. Then
\begin{equation*} y_1^k \equiv cp^{-\alpha} \sum_{i=2}^s a_i x_i^k  \mmod p^\gamma,\end{equation*}
so $cp^{-\alpha} \sum_{i=2}^s a_i x_i^k$ is a $k$th power residue modulo $p^\gamma$, and hence modulo every power of $p$. Thus there exists $x_1$ such that 
\begin{equation*}
x_1^k \equiv  cp^{-\alpha} \sum_{i=2}^s a_i x_i^k  \mmod p^r,
\end{equation*}
which yields the congruence \eqref{pAdicEqn}. We conclude that $T(p)>0$ for every $p$, which shows that $\fG$ is a positive real number.
\end{proof}

Let $K= C \fG$ with $C$ as in \S \ref{CM}. The estimate \eqref{alt} and Lemma \ref{SSaverage} imply Theorem \ref{average}, since $T$ is a positive integer satisfying $B^\varepsilon \ll T \le B^\delta$.

\section{Second moment analysis}
\label{SMA}

We attack Proposition \ref{dream} by focusing on the singular series.

\begin{lemma} \label{attack} Assume equation \eqref{nilseq}. Then there exists a positive-valued arithmetic function $r(t)$, decreasing to 0, such that 
\begin{equation*} | \fG(t) - \fG | < r(t) \end{equation*}
for almost all positive integers $t$.
\end{lemma}

\begin{proof} 
The result is trivial if $\fG(t) = \fG$ for all $t$, so assume otherwise. Let
\[ A_T = T^{-1} \sum_{t \le T} (\fG(t)-\fG)^2.
\]
Define $\Upsilon: \bN \to \bN$ by $\Upsilon(1) =1$ and, for $m \ge 2$,
\begin{equation*}
\Upsilon(m) = \min \{T_0 \in \bN: \text{if } T\ge T_0 \text{ then } A_T \le m^{-3} \}.
\end{equation*}
This is well defined, by our assumption \eqref{nilseq}, and $\Upsilon(m)$ is increasing to infinity, since $A_T$ is positive for large $T$. Define $r(t) = m^{-1}$ if $\Upsilon(m) \le t < \Upsilon(m+1)$.

Now $r(t)$ is a positive-valued arithmetic function, decreasing to 0. Moreover, 
\begin{equation*} T^{-1} \# \{t \le T: |\fG(t) - \fG| \ge r(t) \} \le  r(T)^{-2}A_T \to 0
\end{equation*}
as $T \to \infty$.
\end{proof}

By Lemma \ref{attack} and the bound \eqref{alt}, the following statement holds for almost all positive integers $t$:  if $B \ge t^{1/\delta}$ then
\begin{equation*}
 N(B,t)t^dB^{k-s} - K \ll  r(t) + t^{-\eps/\delta}.
\end{equation*}
Therefore there exists a positive constant $\fc$ such that Proposition \ref{dream} holds with 
$\rho(t) = \fc( r(t) +  t^{-\eps/\del})$.

Our final task is to establish Theorem \ref{PosVar}. Let
\begin{equation} \label{Udef}
U(q_1,q_2) = (q_1q_2)^{-1} \sum_{t \le q_1q_2} \fG(t,q_1)\fG(t,q_2)
\end{equation}
and
\begin{equation} \label{fCdef}
\fC = \sum_{q_1=1}^\infty \sum_{q_2=1}^\infty U(q_1,q_2).
\end{equation}
The bound \eqref{SingularBound1} ensures that the series \eqref{fCdef} converges absolutely, so $\fC \in \bC$. 

\begin{lemma} \label{VarConv}
The constant $\fC$ is a positive real number, and
\begin{equation} \label{VarConvEq}
T^{-1} \sum_{t \le T} \fG(t)^2 = \fC + O(T^{\eps-1}+  T^{\eps + 2 - (s-1)/k}).
\end{equation}
\end{lemma}

\begin{proof}
We begin by establishing the asymptotic formula \eqref{VarConvEq}. Our argument is similar to that used in the proof of Lemma \ref{SSaverage} to establish the asymptotic formula \eqref{SSaverageEq}.  Write
\begin{equation*} 
\fG_{t,q_1,q_2} = \fG(t,q_1)\fG(t,q_2).
\end{equation*}
Recall equations \eqref{SqaDef} and \eqref{GtqDef}, and recall that $T$ is a positive integer. For given positive integers $q_1$ and $q_2$, let $r=r(q_1,q_2,T)$ denote the remainder when $T$ is divided by $q_1q_2$. By periodicity modulo $q_1q_2$, we have
\begin{equation*}
\sum_{t=r+1}^T \fG_{t,q_1,q_2} = \frac{T-r}{q_1q_2} \sum_{t=1}^{q_1q_2}\fG_{t,q_1,q_2}.
\end{equation*}
Recalling equation \eqref{Gt}, we now have
\begin{align*}
T^{-1} \sum_{t \le T} \fG(t)^2 - \fC 
&=
\sum_{q_1=1}^\infty\sum_{q_2=1}^\infty  T^{-1} \Big(
\sum_{t\le T} \fG_{t,q_1,q_2}- \frac T{q_1q_2} \sum_{t=1}^{q_1q_2} \fG_{t,q_1,q_2}
\Big)
\\
&= \sum_{q_1=1}^\infty\sum_{q_2=1}^\infty  T^{-1} \Big(
\sum_{t=1}^r \fG_{t,q_1,q_2}- \frac r{q_1q_2} \sum_{t=1}^{q_1q_2} \fG_{t,q_1,q_2}
\Big). \end{align*}
On noting that $r(q_1,q_2,T) = T$ when $q_1q_2 >T$, the inequality \eqref{SingularBound1} now gives
\begin{align} \notag
T^{-1} \sum_{t \le T} \fG(t)^2 - \fC  \ll & \: T^{-1}\sum_{q_1q_2\le T} r(q_1q_2)^{1+(1-s)/k} 
\\ 	\label{pot}
&+ \sum_{q_1q_2>T} (q_1q_2)^{1+(1-s)/k}.
\end{align}

Using the calculation \eqref{calculation} with $q=q_1q_2$, we have
\begin{equation} \label{firstly}
\sum_{q_1q_2\le T} r(q_1q_2)^{1+(1-s)/k} \ll T^\eps \sum_{q \le T}r q^{1+(1-s)/k} \ll T^\eps + 
T^{\eps + 3 - (s-1)/k}. \end{equation}
Furthermore,
\begin{equation}  \label{secondly} 
\sum_{q_1q_2>T} (q_1q_2)^{1+(1-s)/k} = P_1+ P_2,
\end{equation}
where
\begin{equation} \label{P_1} P_1 = \sum_{q_1 >T} q_1^{1+(1-s)/k} \sum_{q_2 =1}^\infty q_2^{1+(1-s)/k}
\ll T^{2+(1-s)/k} \end{equation}
and
\begin{align} \notag P_2 &= \sum_{ q_1 \le T} q_1^{1+(1-s)/k} \sum_{q_2 > T/q_1} q_2^{1+(1-s)/k} \\
& \label{P_2}
\ll   \sum_{q_1 \le T}  q_1^{1+(1-s)/k}  (T/q_1)^{2+(1-s)/k} \ll T^{\eps+2+(1-s)/k}.
\end{align}

Considering equation \eqref{secondly} and the inequalities \eqref{pot}, \eqref{firstly}, \eqref{P_1} and \eqref{P_2} yields equation \eqref{VarConvEq}. In particular $T^{-1} \sum_{t \le T} \fG(t)^2$ converges to $\fC$ as $T \to \infty$ so, 
by Lemma \ref{SSaverage},
\begin{equation} \label{view}  0 \le T^{-1} \sum_{t\le T} (\fG(t)-\fG)^2 \to \fC - \fG^2, \end{equation}
so $\fC$ is a real number with $\fC \ge \fG^2 > 0$.
\end{proof}

In view of equation \eqref{view}, in order to prove Theorem \ref{PosVar} it remains to show that $\fC > \fG^2$. For positive integers $t$ and $q$, let $M_t(q)$ count solutions $\mathbf{x}$ modulo $q$ to the congruence \eqref{ModEqn}. Then 
\begin{equation}\label{SumMt} M(q) = \sum_{t=1}^q M_t(q).\end{equation}
The Chinese remainder theorem shows that $M_t(\cdot)$ is multiplicative for any given $t$. For positive integers $t$ and prime numbers $p$, define
\begin{equation} \label{tSeries} 
T_t(p) = \sum_{h=0}^\infty \fG(t,p^h).
\end{equation}
The classical theory in \cite[Chapters 5 and 8]{Dav2005} tells us that $T_t(p)$ is a positive real number, that
\begin{equation} \label{MtNote} p^{r(1-s)}M_t(p^r)  = \sum_{h=0}^r \fG(t,p^h), \end{equation}
and that 
\begin{equation} \label{tLimit} T_t(p) = \lim_{r \to \infty} p^{r(1-s)}M_t(p^r). \end{equation}

For positive integers $q_1$ and $q_2$, let $M(q_1,q_2)$ count integer solutions to the system
\begin{equation*}
\sum_{i=1}^s h_i(t) x_i^k \equiv 0 \mmod q_1, \qquad \sum_{i=1}^s h_i(t) y_i^k \equiv 0 \mmod q_2
\end{equation*}
with $1 \le t \le q_1q_2$, $1 \le x_1, \ldots, x_s \le q_1$ and $1 \le y_1, \ldots, y_s \le q_2$. By orthogonality,
\begin{equation*}
M(q_1,q_2)= (q_1q_2)^{-1}\sum_{r_1,r_2,\mathbf{x}, \mathbf{y},t} \prod_{i=1}^s e( q_1^{-1}r_1h_i(t)x_i^k+q_2^{-1}r_2h_i(t)y_i^k) ,\end{equation*}
where the summation is over $1 \le r_1, x_1,\ldots,x_s \le q_1$, $1 \le r_2, y_1,\ldots,y_s \le q_2$ and $1 \le t \le q_1q_2$.  Recall equations \eqref{SqaDef}, \eqref{GtqDef} and \eqref{Udef}. By writing $r_i/q_i = a_i/d_i$ with $d_i >0$ and $(a_i,d_i)=1$ for $i=1,2$, we can now deduce that
\begin{equation} \label{DoubleCong} (q_1q_2)^{-s} M(q_1,q_2) = \sum_{d_1| q_1} \sum_{d_2| q_2} U(d_1,d_2).\end{equation}

By considering the underlying congruences, we observe that
\begin{equation} \label{obs} M(q,q) = q\sum_{t=1}^q M_t(q)^2 \end{equation}
for all $q$. The multiplicativity of $q \mapsto M(q,q)$ is inherited from that of $M_t(\cdot)$, for if $(u,v)=1$ then
\begin{align*} M(uv,uv) &= uv \sum_{t \mmod uv} M_t(u)^2 M_t(v)^2  \\
&= uv \sum_{a=1}^u \sum_{b=1}^v M_{av+bu}(u)^2 M_{av+bu}(v)^2 = M(u,u)M(v,v).
\end{align*}

For each prime $p$, let
\begin{equation} \label{ChiDef}
\chi_p =  \sum_{a=0}^\infty \sum_{b=0}^\infty U(p^a,p^b). 
\end{equation}
We showed that the series \eqref{fCdef} converges absolutely so, \emph{a fortiori}, the series \eqref{ChiDef} converges absolutely. Equation \eqref{DoubleCong} yields 
\[ p^{-2rs}M(p^r,p^r) = \sum_{a=0}^r \sum_{b=0}^r U(p^a,p^b), \] so
\begin{equation} \label{ChiLim}
\chi_p = \lim_{r \to \infty} p^{-2rs}M(p^r,p^r).
\end{equation}
The identities \eqref{TpAlt}, \eqref{SumMt}, \eqref{obs} and \eqref{ChiLim}, together with Cauchy's inequality, yield
\begin{equation} \label{densities2v1}
\chi_p \ge T(p)^2.
\end{equation}

\begin{lemma} We have
\begin{equation} \label{DensityProduct} \fC = \prod_p \chi_p.\end{equation}
\end{lemma}

\begin{proof} Let $R \ge 2$ be a real number, and put 
\[ P = \prod_{p \le R}p.\] The multiplicativity of $q \mapsto M(q,q)$ and the identity \eqref{DoubleCong} yield
\[ \prod_{p \le R} p^{-2rs} M(p^r, p^r) = P^{-2rs}M(P^r,P^r) = \sum_{d_1, d_2 | P^r} U(d_1,d_2),\]
and now equations \eqref{ChiLim} gives
\begin{equation*} \prod_{p \le R} \chi_p = \sum_{d_1,d_2 \in \cS(R)}U(d_1,d_2),
\end{equation*}
where $\cS(R) = \{ n \in \bN: \text{ if } p|n \text{ then } p \le R \}$ is the set of $R$-smooth numbers. Recalling the definitions \eqref{Udef} and \eqref{fCdef}, as well as the inequality \eqref{SingularBound1}, we now see that
\begin{align*} \fC - \prod_{p \le R} \chi_p &= \sum_{(d_1,d_2) \notin \cS(R)^2}  U(d_1,d_2) 
\ll \sum_{(d_1,d_2) \notin \cS(R)^2} (d_1d_2)^{1+(1-s)/k} \\
&\le \sum_{d_1 > R \text{ or } d_2 > R} (d_1d_2)^{1+(1-s)/k} 
 \ll R^{2+(1-s)/k}
\end{align*}
converges to zero as $R \to \infty$.
\end{proof}

It remains to show that $\chi_p > T(p)^2$ for some prime $p$. Then equations \eqref{EulerProd} and \eqref{DensityProduct}, together with the inequality \eqref{densities2v1}, will give $\fC > \fG^2$, which will thereby complete the proof of Theorem \ref{PosVar}.

\begin{lemma} \label{striker} We have
\begin{equation} \label{strikerEq} \chi_p - T(p)^2 = \frac12 \lim_{r \to \infty} p^{-2r} \sum_{1 \le u,v \le p^r} 
(T_u(p)-T_v(p))^2.
\end{equation}
\end{lemma}

\begin{proof} The identities \eqref{TpAlt}, \eqref{SumMt}, \eqref{obs} and \eqref{ChiLim} give
\begin{align} 
\notag \chi_p - T(p)^2
&= \lim_{r \to \infty} p^{-2rs} \Big( p^r \sum_{t \le p^r }M_t(p^r)^2 -  \Big(\sum_{t \le p^r} M_t(p^r)\Big)^2 \Big) \\
\label{FirstBit} &=  \frac12 \lim_{r \to \infty} p^{-2rs}  \sum_{1 \le u,v \le p^r} (M_u(p^r) - M_v(p^r))^2.
\end{align}
For any $t$ and $r$, equations \eqref{tSeries} and \eqref{MtNote}, together with the inequality \eqref{SingularBound1}, yield
\[
T_t(p) - p^{r(1-s)}M_t(p^r) = \sum_{h>r}\fG(t,p^h) \ll 
p^{(1+(1-s)/k)(r+1)}.
\]
The definition \eqref{tSeries} and the bound \eqref{SingularBound1} show that, moreover, 
\[ T_t(p) \ll 1 \]
uniformly in $t$. Hence
\[ (T_u(p)-T_v(p))^2 -  p^{2r(1-s)} (M_u(p^r) - M_v(p^r))^2 \ll p^{(1+(1-s)/k)(r+1)}.\]
This equates the right hand sides of equations \eqref{strikerEq} and \eqref{FirstBit}.
\end{proof}

\begin{lemma} \label{goal}
Suppose there exist $p$, $v$ and $\kappa = \kappa(p) > 0$ such that if $a \equiv 1 \mmod p$ and $b \equiv v \mmod p^2$ then $|T_a(p) - T_b(p)| \ge \kappa$. Then Theorem \ref{PosVar} holds.
\end{lemma}

\begin{proof} By Lemma \ref{striker}, 
\[ \chi_p - T(p)^2 
\ge \frac12 \lim_{r \to \infty}p^{r-1}p^{r-2} \kappa^2 / p^{2r} = \frac12 \kappa^2 /p^3  > 0. \]
\end{proof}

Our methods allow us to compute the $p$-adic density $T_t(p)$ whenever $p$ divides no more than one of $h_1(t), \ldots, h_s(t)$; the argument needs to be modified slightly if $p^k$ divides some $h_i(t)$. We do not require such generality, so we specialise for simplicity. We choose $p$ and $v$ via the following lemma.

\begin{lemma}\label{IrredPoly} Let $f$ be a nonconstant polynomial with integer coefficients, irreducible in $\bQ[x]$. Then there exist infinitely many primes $p$ such that $f$ has a root modulo $p$ that is not a root modulo $p^2$. 
\end{lemma}

\begin{proof}
First we establish that $f$ can be assumed to be monic. Put $d = \deg f$, and let $a$ be the leading coefficient of $f$. Now $a^{d-1}f$ is a monic irreducible polynomial in $y =ax$, with integer coefficients. For large $p$ we know that $p$ does not divide $a$, so applying the result for this polynomial yields the desired outcome.

With $f$ monic, the Frobenius density theorem (see \cite{Fro1896} or \cite{Sur2003}) yields infinitely many primes $p$ for which there exist $a_1,\ldots,a_d \in \bF_p$ such that
\begin{equation*}
f(x) = (x-a_1) \cdots (x-a_d) \in \bF_p[x], 
\end{equation*}
and there are infinitely many such $p$ that do not divide the discriminant $\Delta$ of $f$. Given such a prime $p$, the discriminant is nonzero in $\bF_p$, so $a_i \ne a_j$ whenever $i \ne j$. Now Hensel's lemma allows us to lift the $a_i$ to roots $\tilde{a_i} \in \bZ_p$ of $f$. Thus
\begin{equation*}
f(x) = (x-\tilde{a_1}) \cdots (x-\tilde{a_d}) \in \bZ_p[x],
\end{equation*}
and $\tilde {a_i} \not \equiv \tilde {a_j}$ modulo $p$ whenever $i \ne j$. Put $\tilde v = p+ \tilde{a_1} \in \bZ_p/p^2 \bZ_p$, so that 
\begin{equation*} f( \tilde v) = p \prod_{i=2}^d (p+ \tilde a_1 - \tilde a_i)
\end{equation*}
is exactly divisible by $p$. Now $p \| f(v)$, where $v$ is any integer representing the image of $\tilde v$ under the isomorphism $\bZ_p / p^2 \bZ_p \to \bZ / p^2 \bZ$.
\end{proof}

By Lemma \ref{IrredPoly} we can choose an arbitrarily large prime $p$ and an integer $v$ such that $p \| h_s(v)$. Since $p$ is large it does not divide $h_i(1)$ for any $i$, nor does it divide $k$, and by the inequality \eqref{PolyGCDbound} it does not divide $h_i(v)$ for $i=1,2,\ldots,s-1$. Let $a \equiv 1 \mmod p$ and $b \equiv v \mmod p^2$. 

Put $g = (p-1,k)$, let $t$ be an integer, let $r$ be a large positive integer, and write $a_i = h_i(t)$ for $i=1,2,\ldots,s$; we will consider $t=a$ and $t=b$. Recall equation \eqref{tLimit}, and that $M_t(p^r)$ counts integer solutions to
\begin{equation} \label{MainCong}
a_1x_1^k + \ldots + a_s x_s^k \equiv 0 \mmod p^r
\end{equation}
with $1 \le x_1, \ldots, x_s \le p^r$. Since $\gamma = 1$ in equation \eqref{GammaDef}, a \emph{$k$th power residue} will mean modulo $p$, or equivalently modulo $p^y$ for any positive integer $y$.

\subsection*{Computation of $T_a(p)$}

In this case $p$ does not divide $a_i$ for any $i$. For a given solution $\mathbf{x}$, let $m$ be maximal such that $p^m | x_i$ for all $i$, and let $I$ denote the nonempty subset
\begin{equation*} I = \{ i: p^m \| x_i \}. \end{equation*}
We count solutions with a specific $m$ and $I$, in order to later sum the contributions from all $m$ and $I$. We only need to consider $m < r/k$, for the number of solutions with $p^r | x_s^k$ is $o(p^{r(s-1)})$ as $r \to \infty$. 

For $i \in I$ write $x_i = p^m y_i$. Our congruence \eqref{MainCong} becomes
\begin{equation*} 
 \sum_{i \in I} a_i y_i^k + \sum_{j \notin I} a_j(x_j/p^m)^k    \equiv 0 \mmod p^{r-mk}
\end{equation*}
with the $y_i$ modulo $p^{r-m}$ not divisible by $p$ and the $x_j$ modulo $p^r$ divisible by $p^{m+1}$ for $j \notin I$. Fix $i_0 \in I$ and put $I^* = I \setminus \{i_0\}$. Note that 
\[ g = (\varphi(p^{r-mk}),k),\]
where $\varphi$ is Euler's totient function. Choosing $y_i$ for $i \in I^*$ and $x_j$ for $j \notin I$ determines $g$ solutions $y_{i_0}$ if
\begin{equation*}
-a_{i_0}^{-1} \sum_{i \in I^*} a_i y_i^k
\end{equation*}
is a $k$th power residue and no solutions otherwise. Let $A_I$ count solutions to
\begin{equation*} \sum_{i \in I} a_i z_i^k \equiv 0 \mmod p, \qquad z_i \in (\bZ / p \bZ)^\times.
\end{equation*}
For any choice of $(x_j)_{j \notin I}$ there are $A_I/g$ choices for $(y_i \in (\bZ / p \bZ)^\times)_{i \in I^*}$ yielding solutions $y_{i_0}$, and these lift to $g^{-1}A_I (p^{r-m-1})^{\# I-1}$ choices for $(y_i \mmod p^{r-m})_{i \in I^*}$. 

There are $p^{r-m-1}$ possible $x_j$ for each $j \notin I$, so the total number of solutions with $m < r/k$ is thus
\begin{equation*}
\sum_{ \emptyset \ne I \subseteq\{1,\ldots,s\}} \: \sum_{0 \le m < r/k} (p^{r-m-1})^{s-1} A_I.
\end{equation*}
Hence, by equation \eqref{tLimit},
\begin{align*}
T_a(p) &= \sum_{ \emptyset \ne I \subseteq\{1,\ldots,s\}} A_I \sum_{m=0}^\infty (p^{1-s})^{m+1} = 
\frac{\sum_{ \emptyset \ne I \subseteq\{1,\ldots,s\}} A_I } {p^{s-1}-1} \\
&= \frac{A-1} {p^{s-1}-1},
\end{align*}
where $A$ is the number of solutions $w_1,\ldots,w_s$ modulo $p$ to the congruence
\begin{equation*} \sum_{i=1}^s a_i w_i^k \equiv 0 \mmod p.\end{equation*}
We have simplified $T_a(p)$ sufficiently for our purposes, but note that $A$ and hence $T_a(p)$ is easily computable in any specific instance. 

\subsection*{Computation of $T_b(p)$}

This is similar to the previous calculation. In this case $p \| a_s$, and $p$ does not divide $a_i$ for $i=1,2,\ldots,s-1$. Let $m < r/k$ be maximal such that $p^m | x_i$ for $i=1,2,\ldots,s-1$. Note that $p^m$ necessarily divides $x_s$, since $p^2$ does not divide $a_s$. Now put
\begin{equation*} I = \{ i \in \{1,\ldots, s-1 \}: p^m \| x_i \} \end{equation*}
and $J = \{1,\ldots, s-1 \} \setminus I$.
Our congruence \eqref{MainCong} becomes
\begin{equation*}
 \sum_{i \in I} a_i y_i^k + \sum_{j \notin I} a_j (x_j/p^m)^k   \equiv 0 \mmod p^{r-mk},
\end{equation*}
with the $y_i \mmod p^{r-m}$ not divisible by $p$, the $x_j \mmod p^r$ divisible by $p^{m+1}$ for $j \in J$, and $x_s \mmod p^r$ divisible by $p^m$. Note that the second summation above is divisible by $p$, since $a_s$ is divisible by $p$. In analogy with the computation of $T_a(p)$, we deduce that
\begin{equation*} T_b(p) = \frac{p(\tilde A -1)}{p^{s-1}-1}, \end{equation*}
where $\tilde A$ is the number of solutions $w_1,\ldots, w_{s-1}$ modulo $p$ to the congruence
\begin{equation*} \sum_{i=1}^{s-1} a_i w_i^k \equiv 0 \mmod p.\end{equation*}
Note that $\tilde A$ and hence $T_b(p)$ is easily computable in any specific instance. 

The integer $T_b(p) \cdot (p^{s-1}-1)$ is divisible by $p$, while $T_a(p) \cdot (p^{s-1}-1)$ is not divisible by $p$, since the Chevalley-Warning theorem (see \cite{Ax1964}) yields $p | A$. In particular, the hypotheses of Lemma \ref{goal} are satisfied with $\kappa = (p^{s-1}-1)^{-1}$. This completes the proof of Theorem \ref{PosVar}.

\bibliographystyle{amsbracket}

\begin{thebibliography}{50}

\bibitem{Ax1964}
J. Ax, \emph{Zeroes of polynomials over finite fields}, Amer. J. Math. \textbf{86} (1964), 255--261.

\bibitem{Bre2004}
T. Breyer, \emph{\"Uber Hasseprinzipien von Diagonalformen}, Ph.D. Thesis, Universit\"at Stuttgart, Shaker, Aachen, 2004.

\bibitem{BroD2009}
T. D. Browning and R. Dietmann, \emph{Solubility of Fermat equations}, Quadratic forms---algebra, arithmetic and geometry, Contemp. Math. \textbf{493} (2009), 99--106.

\bibitem{Bru1996}
J. Br\"udern, \emph{Cubic Diophantine inequalities, II}, J. London Math. Soc. (2) 53 (1996), no. 1, 1--18.

\bibitem{BruD2012ineq}
J. Br\"udern and R. Dietmann, \emph{Random Diophantine inequalities of additive type}, Adv. Math. \textbf{229} (2012), no. 6, 3079--3095.

\bibitem{BruD2012eq}
J. Br\"udern and R. Dietmann, \emph{Random Diophantine equations, I}, arXiv:1212.4800v1.

\bibitem{Dav2005}
H. Davenport, \emph{Analytic methods for Diophantine equations and Diophantine inequalities}, 2nd edition, Cambridge University Press, Cambridge, 2005.

\bibitem{DL1963}
H. Davenport and D. J. Lewis, \emph{Homogeneous additive equations}, Proc. Roy. Soc. London Ser. A, \textbf{274} (1963), 443--660.

\bibitem{Fro1896} 
F. G. Frobenius, \emph{\"Uber Beziehungen zwischen den Primidealen eines algebraischen K\"orpers und den Substitutionen seiner Gruppe}, Sitz. Akad. Wiss. Berlin (1896), 689--703 (= Ges. Abh., II, 719--733).

\bibitem{HB1996}
D. R. Heath-Brown, \emph{A new form of the circle method, and its application to quadratic
forms}, J. Reine Angew. Math. \textbf{481} (1996), 149--206.

\bibitem{Klo1927}
H. D. Kloosterman, \emph{On the representation of numbers in the form $ax^2+by^2+cz^2+dt^2$},
Acta Math. \textbf{49} (1927), 407--464.

\bibitem{Kub2009}
G. Kuba, \emph{On the distribution of reducible polynomials}, Math. Slovaca \textbf{59} (2009), no. 3, 349--356.

\bibitem{Lew1957}
D. J. Lewis, \emph{Cubic congruences}, Michigan Math J. \textbf{4} (1957), 85--95.

\bibitem{Mad2010}
A. Madlener, \emph{Diagonale Formen mit polynomialen Koeffizienten und Summen bin\"arer Formen}, Ph.D. Thesis, Universit\"at Stuttgart, 2010.

\bibitem{PS1976}
G. P\'olya and G. Szeg\"o, \emph{Problems and Theorems in Analysis, Vol. II}, Springer, Berlin-
Heidelberg-New York, 1976.

\bibitem{PV2003}
B. Poonen and J. F. Voloch, \emph{Random Diophantine equations}, Arithmetic of higher-dimensional algebraic varieties, Prog. Math. \textbf{226} (2004), 175--184.

\bibitem{Sch1979}
W. M. Schmidt, \emph{Small zeros of additive forms in many variables. II}, Acta Math. \textbf{143} (1979), no. 3--4, 219--232.

\bibitem{Sur2003} 
B. Sury, \emph{Frobenius and his density theorem for primes}, Resonance \textbf{8} (2003), no. 12, 33--41.

\bibitem{Vau1997}
R. C. Vaughan, \emph{The Hardy-Littlewood method}, 2nd edition, Cambridge University Press, Cambridge, 1997.

\bibitem{Woo2013}
T. D. Wooley, \emph{Vinogradov's mean value theorem via efficient congruencing, II}, Duke Math. J. \textbf{162} (2013), no. 4, 673--730.








\end{thebibliography}
\providecommand{\bysame}{\leavevmode\hbox to3em{\hrulefill}\thinspace}

\end{document}